\documentclass[11pt]{article}
\usepackage[pagebackref,linktocpage=true,colorlinks=true,linkcolor=Blue,citecolor=BrickRed,urlcolor=RoyalBlue]{hyperref}
\usepackage[alphabetic,msc-links,abbrev]{amsrefs} 

\usepackage[a4paper,margin=1in]{geometry}
\usepackage[UKenglish]{babel}
\usepackage[T1]{fontenc}
\usepackage{lmodern}
\usepackage{microtype}

\usepackage{authblk}

\usepackage{amsmath,amssymb,amsthm,mathtools}
\usepackage{mathrsfs}
\usepackage{graphicx}
\usepackage{tikz}
\usepackage{pgfplots}
\usepackage{hyperref}
\usepackage[nameinlink,noabbrev]{cleveref}

\pgfplotsset{compat=1.18}
\usetikzlibrary{arrows.meta}

\hypersetup{
  colorlinks=true,
  linkcolor=blue,
  citecolor=blue,
  urlcolor=blue
}

\newtheorem{theorem}{Theorem}[section]
\newtheorem{proposition}[theorem]{Proposition}
\newtheorem{lemma}[theorem]{Lemma}
\newtheorem{corollary}[theorem]{Corollary}
\theoremstyle{definition}
\newtheorem{definition}[theorem]{Definition}
\theoremstyle{remark}
\newtheorem{remark}[theorem]{Remark}

\newcommand{\R}{\mathbb R}
\newcommand{\T}{\mathbb T}
\newcommand{\Z}{\mathbb Z}
\newcommand{\pv}{\operatorname{p.v.}}
\newcommand{\sign}{\operatorname{sign}}

\title{On the weak formulation of Prandtl's minimum drag problem
}
\author{Aram L. Karakhanyan}
\author{Yigit Katgi}
\affil{School of Mathematics, The University of Edinburgh, Peter Tait Guthrie Road, EH9 3FD, Edinburgh, UK}
\date{}

\begin{document}

\maketitle
\footnote{Corresponding author's email: aram6k@gmail.com}
\footnote{Keywords: Pseudo-differential equations, Sobolev spaces, Prandtl's induced drag formula, variational problem.}
\footnote{AMS Classification: 76G25, 49J27, 76N25}
\begin{abstract}
We study Prandtl's classical problem  on  minimising the induced drag  for a finite wing with fixed span.
The induced drag is given by a singular quadratic functional of the circulation, with admissible functions  satisfying the 
prescribed lift and second-moment conditions. We formulate the problem in the fractional
Sobolev space \(H^{1/2}\), which is the natural energy space for the  functional, prove existence and uniqueness of minimisers by variational methods,
and derive the corresponding Euler--Lagrange equation. 
Passing to a periodic formulation on the
one-dimensional torus, we identify the drag functional with the quadratic form of the
half-Laplacian and solve the resulting singular integral
equation explicitly and recover Prandtl's bell-shaped circulation profile.
\end{abstract}

\section{Introduction}

In Prandtl's lifting-line theory a finite wing is described by a spanwise circulation
distribution
\[
\Gamma:(-1,1)\to\R.
\]
The induced drag is expressed as 
\[
J(\Gamma)=
\frac{1}{4\pi}\pv\!\int_{-1}^1\!\int_{-1}^1
\frac{\Gamma(x)\Gamma'(y)}{x-y}\,dx\,dy,
\]
which can be derived from the  Biot--Savart law. In the classical treatment, one considers smooth circulation laws
and manipulates the principal value integral directly. If one  formally derived the Euler--Lagrange
equations then by solving them one obtains the family of profiles
\[
\Gamma(x)=(a+bx^2)\sqrt{1-x^2},
\]
which play a distinguished role in aerodynamic design.


This solution can be found in the classical paper Prandtl \cite{Prandtl1933}. 
Before Prandtl's work  Munk \cite{Munk} considered a similar problem with the only constraint being the prescribed flux momentum and pointwise boundary conditions imposed on the admissible functions.
His main result states that in order to achieve minimal drag the downwash velocity must be constant. 
Note that for $H^{1/2}$ the trace operator is not defined, hence in the natural energy space, which as it turns out is the Sobolev space $H^{1/2}$, no boundary condition can be imposed on admissible functions.

In a more recent treatment of the problem, Ozanski used Chebyshev polynomials in some weighted $L^2$ space with hyperbolic metric and a relaxed constraint on the  second moment to give an existence proof \cite{Ozanski}, where 
the second moment condition is replaced with a weaker constraint. 

For some  numerical results  see \cite{Panaro}, where the authors observe that it is enough to minimise the functional  over the class of the solutions to the Euler Lagrange equations. They do not impose the second moment constraint.

The main difficulty the authors deal with is to ensure that suitable perturbations of stationary points
still satisfy the constraints and belong to the admissible class $A$. Thus, choosing the natural class $A$ is the key to solve the problem. Note that in aerodynamic literature \cite{Phillips} the formula $J(\Gamma)= \sum_{k\in \mathbb Z} k |a_k|^2$ is obtained through a formal computation, modulo a constant multiplier,  if $\Gamma=\sum_{k\in \mathbb Z} a_ke^{ikx}$. However, it seems that no rigorous mathematical proof of the existence of a minimiser has been given for the minimal induced drag problem.

The purpose of this paper is to give a rigorous variational treatment of this problem in a
fractional Sobolev setting. More precisely, we give a weak formulation of  the induced drag problem on \(H^{1/2}\),
impose the lift and second-moment constraints as continuous linear conditions, and prove
existence and uniqueness of a minimiser by the direct method of calculus of variations \cite{Evans}. We then derive the weak
Euler--Lagrange equation and show that the minimiser satisfies a classical singular integral
equation on \((-1,1)\). Solving this equation by means of Tricomi's inversion formula yields the
explicit structure of the minimiser.

A key step is to pass to a periodic formulation on the one-dimensional torus
\[
\T=\R/2\pi\Z.
\]
In this setting the Biot--Savart operator becomes a Fourier multiplier with symbol \(|k|\), and
the induced drag functional takes the form
\[
J(\Gamma)=\langle (-\Delta)^{1/2}\Gamma,\Gamma\rangle.
\]
This gives access to standard Hilbert space methods while preserving the structure of Prandtl's
original problem.

The problem has a nonlocal character due to the singular kernel. However, the Fourier
representation of the energy reveals that the natural space is \(H^{1/2}\), and in this
framework the minimisation problem becomes both well posed and amenable to analysis. The gain
from the torus formulation is therefore technical rather than conceptual: it furnishes a clean
spectral description of the energy, after which the original interval problem can be recovered
by rescaling.

\medskip

\noindent\textbf{Main results.}
The first result gives existence and uniqueness of a minimiser.

\begin{theorem}\label{thm:exist-unique}
Let \(L_0,M_0\in\R\). Then there exists a unique minimiser
\(\Gamma_\ast\in\mathcal A\) of the functional
\[
J(\Gamma)=\langle (-\Delta)^{1/2}\Gamma,\Gamma\rangle
\]
over the admissible class
\[
\mathcal A
:=
\Bigl\{
\Gamma\in H^{1/2}(\T):
\int_{-\pi}^{\pi}\Gamma(x)\,dx=L_0,\ 
\int_{-\pi}^{\pi}x^2\Gamma(x)\,dx=M_0
\Bigr\}.
\]
\end{theorem}
Here $(-\Delta)^{1/2}$ is the half-Laplacian defined as a pseudo-differential operator.
The minimiser satisfies a weak Euler--Lagrange equation.

\begin{theorem}\label{thm:EL}
Let \(\Gamma_\ast\) be the minimiser in Theorem~\ref{thm:exist-unique}. Then there exist
constants \(\lambda_0,\lambda_2\in\R\) such that
\[
(-\Delta)^{1/2}\Gamma_\ast=\lambda_0+\lambda_2x^2
\qquad\text{in }H^{-1/2}(\T).
\]
\end{theorem}

The next result identifies the regularity available on the torus.

\begin{theorem}\label{thm:regularity}
Let \(\Gamma_\ast\) be the minimiser. Then
\[
\Gamma_\ast\in H^s(\T)\qquad\text{for every }s<\frac52.
\]
In particular,
\[
\Gamma_\ast\in C^{1,\alpha}(\T)
\qquad\text{for every }\alpha<1.
\]
\end{theorem}

Finally, after returning to the interval \((-1,1)\), we recover the explicit form of the
solution.

\begin{theorem}\label{thm:explicit}
Let \(\Gamma_\ast\) be the unique minimiser, rescaled to \((-1,1)\). Then there exist constants
\(c_0,a,b\in\R\) such that
\[
\Gamma_\ast(x)=c_0+(a+bx^2)\sqrt{1-x^2},
\qquad x\in(-1,1).
\]
If, in addition, one imposes the physical tip condition \(\Gamma_\ast(\pm1)=0\), then
\[
\Gamma_\ast(x)=(a+bx^2)\sqrt{1-x^2}.
\]
\end{theorem}

The paper is organised as follows. In Section~2 we recall the classical lifting-line
formulation. In Section~3 we introduce the periodic Sobolev framework and identify the drag
with the quadratic form of the half-Laplacian. Section~4 contains the proof of existence and
uniqueness. In Section~5 we derive the Euler--Lagrange equation and obtain regularity of the
minimiser. Finally, in Section~6 we return to the interval \((-1,1)\), derive the classical
singular integral equation, and solve it explicitly.

\section{Classical formulation}

In Prandtl's lifting-line theory, a finite wing is modelled by a spanwise circulation
distribution \(\Gamma:(-1,1)\to\R\). The induced downwash is given formally by
\begin{equation}\label{eq:BS-classical}
w(x)=\frac{1}{4\pi}\pv\!\int_{-1}^1\frac{\Gamma'(y)}{x-y}\,dy,
\qquad x\in(-1,1),
\end{equation}
where the principal value is needed because of the singularity at \(y=x\).

\begin{figure}[h]
\centering
\begin{tikzpicture}[scale=1.15,>=Latex]

\foreach \y in {-1.4,-0.5,0.4}
{
    \draw[->] (-3.3,\y) -- (-2.1,\y);
}
\node at (-2.7,1.0) {$V_\infty$};

\draw[thick] (-1.2,-2.4) -- (0.9,2.4);
\node[rotate=65] at (-0.6,0.15) {lifting line};

\node[left] at (-1.2,-2.4) {$-1$};
\node[above] at (0.60,2.30) {$1$};

\node at (1.35,2.75) {$\Gamma(x)$};

\foreach \y in {-2.1,-1.7,-1.3,-0.9,-0.5,-0.1,0.3,0.7,1.1,1.5,1.9,2.3}
{
    \pgfmathsetmacro{\xstart}{-1.2 + ((\y+2.4)/4.8)*2.1}
    \draw[thin] (\xstart,\y) -- (5.2,\y);
}

\foreach \y in {-1.7,-0.9,-0.1,0.7,1.5}
{
    \draw[->] (4.1,\y) -- (4.9,\y);
}

\foreach \t in {0.15,0.35,0.55,0.75,0.9}
{
    \draw[->] ({-1.2 + \t*2.1},{-2.4 + \t*4.8}) -- ++(0,-0.55);
}

\node at (0.8,0.5) {$w(x)$};

\end{tikzpicture}
\caption{Schematic of Prandtl's lifting-line model. The circulation \(\Gamma(x)\)
is supported on the lifting line, while the trailing vortex sheet is convected
downstream by the freestream \(V_\infty\) and induces a downwash \(w(x)\).}
\label{fig:liftingline}
\end{figure}
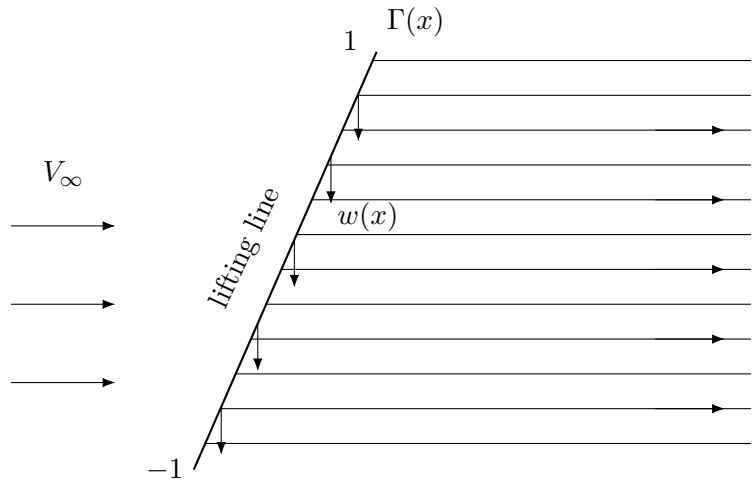

By the Kutta--Joukowski law, the induced drag is obtained by integrating the product of the
circulation and the downwash along the span. Up to a constant factor, this leads to the
quadratic functional
\begin{equation}\label{eq:J-classical}
J_{\mathrm{cl}}(\Gamma)
:=
\int_{-1}^1 \Gamma(x)\,w(x)\,dx
=
\frac{1}{4\pi}\pv\!\int_{-1}^1\!\int_{-1}^1
\frac{\Gamma(x)\Gamma'(y)}{x-y}\,dx\,dy,
\end{equation}
which is well defined for smooth compactly supported circulation laws, see  \cite{Prandtl1933} and \cite{prandtl-original}.

Prandtl's problem is to minimise \(J_{\mathrm{cl}}\) subject to linear constraints. In this paper
we prescribe the total lift
\[
\int_{-1}^1\Gamma(x)\,dx=L_0
\]
and the second moment
\[
\int_{-1}^1x^2\Gamma(x)\,dx=M_0.
\]
These constraints correspond to fixing the total loading and an effective spanwise moment of the
circulation distribution.

The classical formulation is natural from the physical point of view, but for a rigorous
variational treatment it is preferable to work in a weaker functional setting.

\section{Periodic Sobolev formulation}

We reformulate the problem on the torus
\[
\T=\R/2\pi\Z,
\]
identified with \((-\pi,\pi)\) with periodic boundary conditions. We use the Sobolev spaces
\(H^s(\T)\), their duals \(H^{-s}(\T)\) \cite{hebey}, \cite{brezis},  and the fractional Laplacian defined through Fourier
series \cite{katznelson}.

\begin{definition}
For $s\in\mathbb R$, the periodic Sobolev space $H^s(\mathbb T)$ is defined by
\[
H^s(\mathbb T)
:=
\Bigl\{
f\in\mathcal D'(\mathbb T):
\sum_{k\in\mathbb Z}(1+|k|^2)^s |\hat f(k)|^2<\infty
\Bigr\}.
\]
It is equipped with the norm
\[
\|f\|_{H^s(\mathbb T)}^2
:=
\sum_{k\in\mathbb Z}(1+|k|^2)^s |\hat f(k)|^2.
\]
\end{definition}
For $f\in L^2(\mathbb T)$, its Fourier coefficients are defined by
\[
\hat f(k):=\frac{1}{2\pi}\int_{-\pi}^{\pi}f(x)e^{-ikx}\,dx,
\qquad k\in\mathbb Z.
\]
The complex exponentials $e^{ikx}$, $k\in\mathbb Z$, form the natural basis
for Fourier analysis on the torus.

\begin{definition}
For $\alpha\ge0$, the fractional Laplacian on $\mathbb T$ is defined by
\[
\widehat{(-\Delta)^\alpha f}(k)
:=
|k|^{2\alpha}\hat f(k),
\qquad k\in\mathbb Z.
\]
\end{definition}

For \(u\in H^{1/2}(\T)\) one has
\begin{equation}\label{eq:duality-norm-fourier}
\langle (-\Delta)^{1/2}u,u\rangle
=
\|(-\Delta)^{1/4}u\|_{L^2(\T)}^2
=
2\pi\sum_{k\in\Z}|k|\,|\hat u(k)|^2.
\end{equation}
This identity motivates the following definition.

\begin{definition}\label{def:periodic-problem}
Let \(L_0,M_0\in\R\). Define
\[
\ell_0(\Gamma):=\int_{-\pi}^{\pi}\Gamma(x)\,dx,
\qquad
\ell_2(\Gamma):=\int_{-\pi}^{\pi}x^2\Gamma(x)\,dx,
\]
and the admissible set
\[
\mathcal A
:=
\Bigl\{
\Gamma\in H^{1/2}(\T):
\ell_0(\Gamma)=L_0,\ 
\ell_2(\Gamma)=M_0
\Bigr\}.
\]
The periodic Prandtl problem consists in minimising
\begin{equation}\label{eq:J-periodic}
J(\Gamma):=\langle (-\Delta)^{1/2}\Gamma,\Gamma\rangle
\end{equation}
over \(\mathcal A\).
\end{definition}

\begin{remark}
By \eqref{eq:duality-norm-fourier},
\[
J(\Gamma)=\|(-\Delta)^{1/4}\Gamma\|_{L^2(\T)}^2
=
2\pi\sum_{k\in\Z}|k|\,|\hat\Gamma(k)|^2.
\]
In particular, \(J\) is nonnegative.
\end{remark}

\begin{remark}
Since \(H^{1/2}(\T)\hookrightarrow L^2(\T)\) and \(1,x^2\in L^2(-\pi,\pi)\), the functionals
\[
\ell_0(\Gamma)=\int_{-\pi}^{\pi}\Gamma(x)\,dx,
\qquad
\ell_2(\Gamma)=\int_{-\pi}^{\pi}x^2\Gamma(x)\,dx
\]
are bounded on \(L^2(-\pi,\pi)\) by the Cauchy--Schwarz inequality. Hence they define continuous
linear functionals on \(H^{1/2}(\T)\).
\end{remark}

We next record the symmetry of the half-Laplacian with respect to the
\(H^{-1/2}\)--\(H^{1/2}\) pairing.

\begin{lemma}[Symmetry]\label{lem:half-lap-symmetric}
For all \(f,g\in H^{1/2}(\T)\),
\[
\langle (-\Delta)^{1/2}f,g\rangle_{H^{-1/2},H^{1/2}}
=
\langle (-\Delta)^{1/2}g,f\rangle_{H^{-1/2},H^{1/2}}.
\]
\end{lemma}

\begin{proof}
By the Fourier definitions of the pairing and of \((-\Delta)^{1/2}\),
\[
\langle (-\Delta)^{1/2}f,g\rangle
=
2\pi\sum_{k\in\Z}|k|\,\hat f(k)\,\overline{\hat g(k)}
=
2\pi\sum_{k\in\Z}|k|\,\hat g(k)\,\overline{\hat f(k)}
=
\langle (-\Delta)^{1/2}g,f\rangle.
\]
\end{proof}

We now identify the functional \(J\) with the classical drag.

\begin{proposition}\label{prop:J-classical-to-spectral}
For \(\Gamma\in C^\infty(\T)\), define
\[
(B\Gamma)(x):=\pv\!\int_{-\pi}^{\pi}\frac{\Gamma'(y)}{x-y}\,dy,
\qquad x\in(-\pi,\pi),
\]
and
\[
J_{\mathrm{cl}}(\Gamma):=\int_{-\pi}^{\pi}\Gamma(x)(B\Gamma)(x)\,dx.
\]
Then
\begin{equation}\label{eq:J-Fourier}
J_{\mathrm{cl}}(\Gamma)=\pi J(\Gamma).
\end{equation}
In particular, \eqref{eq:J-periodic} provides a canonical extension of Prandtl's induced drag
functional to \(H^{1/2}(\T)\).
\end{proposition}

\begin{proof}
Write
\[
\Gamma(x)=\sum_{k\in\Z}\hat\Gamma(k)e^{ikx},
\qquad
\hat\Gamma(k)=\frac1{2\pi}\int_{-\pi}^{\pi}\Gamma(x)e^{-ikx}\,dx.
\]
Then
\[
\Gamma'(x)=\sum_{k\in\Z}(ik)\hat\Gamma(k)e^{ikx}.
\]

Let
\[
K(t):=\pv\!\Bigl(\frac1t\Bigr),
\qquad t\in(-\pi,\pi),
\]
viewed as a periodic distribution on \(\T\). Then
\[
B\Gamma=\Gamma'*K
\]
in the distributional sense, and therefore
\[
\widehat{B\Gamma}(k)=\widehat{\Gamma'}(k)\,\widehat K(k).
\]
Since
\[
\widehat{\Gamma'}(k)=ik\,\hat\Gamma(k)
\qquad\text{and}\qquad
\widehat K(k)=-i\pi\,\sign(k),
\]
we obtain
\[
\widehat{B\Gamma}(k)=\pi|k|\,\hat\Gamma(k).
\]

By Parseval's identity,
\[
J_{\mathrm{cl}}(\Gamma)
=
\int_{-\pi}^{\pi}\Gamma(x)(B\Gamma)(x)\,dx
=
2\pi\sum_{k\in\Z}\hat\Gamma(k)\,\overline{\widehat{B\Gamma}(k)}
=
2\pi^2\sum_{k\in\Z}|k|\,|\hat\Gamma(k)|^2.
\]
Using \eqref{eq:duality-norm-fourier}, this becomes
\[
J_{\mathrm{cl}}(\Gamma)=\pi J(\Gamma),
\]
as claimed.
\end{proof}

\begin{remark}
The passage from \((-1,1)\) to \((-\pi,\pi)\) is purely a normalization. A linear change of
variables rescales the interval and introduces only a constant factor in the kernel and in the
energy. Such factors do not affect the structure of the variational problem under fixed linear
constraints, nor the form of the minimiser.
\end{remark}

\begin{proposition}\label{prop:J-finite}
The functional \(J\) is finite on \(H^{1/2}(\T)\).
\end{proposition}

\begin{proof}
Let \(\Gamma\in H^{1/2}(\T)\). By the Fourier multiplier definition of \((-\Delta)^{1/4}\) and
Parseval's identity,
\[
J(\Gamma)=\|(-\Delta)^{1/4}\Gamma\|_{L^2(\T)}^2
=
2\pi\sum_{k\in\Z}|k|\,|\hat\Gamma(k)|^2.
\]
Since
\[
|k|\le (1+|k|^2)^{1/2},
\]
we obtain
\[
J(\Gamma)
\le
2\pi\sum_{k\in\Z}(1+|k|^2)^{1/2}|\hat\Gamma(k)|^2
=
2\pi\|\Gamma\|_{H^{1/2}(\T)}^2<\infty.
\]
\end{proof}

\section{Existence and uniqueness of minimisers}

We begin with the nonemptiness of the admissible set.

\begin{lemma}\label{lem:A-nonempty}
For every \(L_0,M_0\in\R\), the set \(\mathcal A\) is nonempty.
\end{lemma}

\begin{proof}
Consider the even \(2\pi\)-periodic function
\[
\Gamma(x)=a+b\cos x,
\qquad x\in(-\pi,\pi),
\]
where \(a,b\in\R\) are to be chosen. Since \(\Gamma\in C^\infty(\T)\), one has
\(\Gamma\in H^{1/2}(\T)\).

First,
\[
\ell_0(\Gamma)=\int_{-\pi}^{\pi}(a+b\cos x)\,dx=2\pi a,
\]
hence
\[
a=\frac{L_0}{2\pi}.
\]

Next,
\[
\ell_2(\Gamma)
=
\int_{-\pi}^{\pi}x^2(a+b\cos x)\,dx
=
a\int_{-\pi}^{\pi}x^2\,dx+b\int_{-\pi}^{\pi}x^2\cos x\,dx.
\]
We compute
\[
\int_{-\pi}^{\pi}x^2\,dx=\frac{2\pi^3}{3},
\qquad
\int_{-\pi}^{\pi}x^2\cos x\,dx=-4\pi.
\]
Therefore
\[
\ell_2(\Gamma)=a\frac{2\pi^3}{3}-4\pi b.
\]
Substituting \(a=L_0/(2\pi)\) and imposing \(\ell_2(\Gamma)=M_0\), we obtain
\[
M_0=\frac{L_0\pi^2}{3}-4\pi b,
\]
so that
\[
b=\frac{\frac{L_0\pi^2}{3}-M_0}{4\pi}.
\]
Thus \(\Gamma\in\mathcal A\), and the admissible set is nonempty.
\end{proof}

We can now prove the main variational result.

\begin{proof}[Proof of Theorem~\ref{thm:exist-unique}]
By Lemma~\ref{lem:A-nonempty}, the admissible set is nonempty. Since
\[
J(\Gamma)=\|(-\Delta)^{1/4}\Gamma\|_{L^2(\T)}^2\ge 0,
\]
the functional is bounded below on \(\mathcal A\). Let \((\Gamma_n)\subset\mathcal A\) be a
minimising sequence, so that
\[
J(\Gamma_n)\to \inf_{\Gamma\in\mathcal A}J(\Gamma).
\]
In particular, there exists \(C_0>0\) such that \(J(\Gamma_n)\le C_0\) for all \(n\).

\smallskip

\noindent\textit{Step 1. Boundedness of the minimising sequence in \(H^{1/2}(\T)\).}
Recall that
\[
\|\Gamma_n\|_{H^{1/2}}^2
=
\sum_{k\in\Z}(1+|k|^2)^{1/2}|\widehat{\Gamma_n}(k)|^2
=
|\widehat{\Gamma_n}(0)|^2+
\sum_{k\neq 0}(1+|k|^2)^{1/2}|\widehat{\Gamma_n}(k)|^2.
\]
For \(k\neq 0\),
\[
(1+|k|^2)^{1/2}\le \sqrt2\,|k|,
\]
hence
\[
\sum_{k\neq 0}(1+|k|^2)^{1/2}|\widehat{\Gamma_n}(k)|^2
\le
\sqrt2\sum_{k\in\Z}|k|\,|\widehat{\Gamma_n}(k)|^2.
\]
Using \eqref{eq:duality-norm-fourier},
\[
J(\Gamma_n)=2\pi\sum_{k\in\Z}|k|\,|\widehat{\Gamma_n}(k)|^2,
\]
so
\[
\sum_{k\neq 0}(1+|k|^2)^{1/2}|\widehat{\Gamma_n}(k)|^2
\le
\frac{\sqrt2}{2\pi}C_0.
\]

It remains to control the zero mode. Since \(\Gamma_n\in\mathcal A\),
\[
\ell_0(\Gamma_n)=\int_{-\pi}^{\pi}\Gamma_n(x)\,dx=L_0.
\]
Hence
\[
\widehat{\Gamma_n}(0)=\frac1{2\pi}\int_{-\pi}^{\pi}\Gamma_n(x)\,dx=\frac{L_0}{2\pi},
\]
and therefore
\[
|\widehat{\Gamma_n}(0)|^2=\frac{L_0^2}{4\pi^2}.
\]
Combining these estimates, we find
\[
\|\Gamma_n\|_{H^{1/2}}^2
\le
\frac{L_0^2}{4\pi^2}+\frac{\sqrt2}{2\pi}C_0.
\]
Thus \((\Gamma_n)\) is bounded in \(H^{1/2}(\T)\).

\smallskip

\noindent\textit{Step 2. Weak compactness and passage to the constraints.}
Since \(H^{1/2}(\T)\) is a Hilbert space, after passing to a subsequence we may assume
\[
\Gamma_n\rightharpoonup \Gamma_\ast
\qquad\text{weakly in }H^{1/2}(\T).
\]
Because \(\ell_0\) and \(\ell_2\) are continuous linear functionals on \(H^{1/2}(\T)\), we may
pass to the limit in the constraints:
\[
\ell_0(\Gamma_\ast)=\lim_{n\to\infty}\ell_0(\Gamma_n)=L_0,
\qquad
\ell_2(\Gamma_\ast)=\lim_{n\to\infty}\ell_2(\Gamma_n)=M_0.
\]
Hence \(\Gamma_\ast\in\mathcal A\).

\smallskip

\noindent\textit{Step 3. Weak lower semicontinuity of \(J\).}
Define
\[
T:=(-\Delta)^{1/4}:H^{1/2}(\T)\to L^2(\T).
\]
By \eqref{eq:duality-norm-fourier},
\[
\|T\Gamma\|_{L^2(\T)}^2
=
2\pi\sum_{k\in\Z}|k|\,|\hat\Gamma(k)|^2
\le
2\pi\sum_{k\in\Z}(1+|k|^2)^{1/2}|\hat\Gamma(k)|^2
=
2\pi\|\Gamma\|_{H^{1/2}(\T)}^2,
\]
so \(T\) is bounded. Moreover,
\[
J(\Gamma)=\|T\Gamma\|_{L^2(\T)}^2.
\]

Since \(\Gamma_n\rightharpoonup\Gamma_\ast\) weakly in \(H^{1/2}(\T)\) and \(T\) is bounded and
linear, we have
\[
T\Gamma_n\rightharpoonup T\Gamma_\ast
\qquad\text{weakly in }L^2(\T).
\]
The \(L^2\)-norm is weakly lower semicontinuous, hence
\[
J(\Gamma_\ast)=\|T\Gamma_\ast\|_{L^2(\T)}^2
\le
\liminf_{n\to\infty}\|T\Gamma_n\|_{L^2(\T)}^2
=
\liminf_{n\to\infty}J(\Gamma_n).
\]
Therefore \(\Gamma_\ast\) is a minimiser.

\smallskip

\noindent\textit{Step 4. Uniqueness.}
Let \(\Gamma_1,\Gamma_2\in\mathcal A\) be minimisers and set
\[
w:=\Gamma_2-\Gamma_1.
\]
Since the constraints are linear, \(\mathcal A\) is affine. Hence
\[
\Gamma_1+tw\in\mathcal A
\qquad\text{for every }t\in\R.
\]
Using bilinearity and Lemma~\ref{lem:half-lap-symmetric},
\[
J(\Gamma_1+tw)
=
J(\Gamma_1)+2t\langle (-\Delta)^{1/2}\Gamma_1,w\rangle+t^2J(w).
\]
Since \(\Gamma_1\) is a minimiser and \(t=0\) is a minimum of the function
\(t\mapsto J(\Gamma_1+tw)\), its derivative at \(0\) vanishes. Therefore
\[
\langle (-\Delta)^{1/2}\Gamma_1,w\rangle=0.
\]
Thus
\[
J(\Gamma_1+tw)=J(\Gamma_1)+t^2J(w).
\]
Taking \(t=1\), we obtain
\[
J(\Gamma_2)=J(\Gamma_1)+J(w).
\]
Since both \(\Gamma_1\) and \(\Gamma_2\) are minimisers, it follows that \(J(w)=0\).

By \eqref{eq:duality-norm-fourier},
\[
J(w)=2\pi\sum_{k\in\Z}|k|\,|\widehat w(k)|^2,
\]
hence \(\widehat w(k)=0\) for all \(k\neq 0\). Moreover,
\[
\ell_0(w)=0,
\]
so
\[
\widehat w(0)=\frac1{2\pi}\int_{-\pi}^{\pi}w(x)\,dx=0.
\]
Therefore all Fourier coefficients of \(w\) vanish, and hence \(w=0\). Thus
\(\Gamma_1=\Gamma_2\), proving uniqueness.
\end{proof}

\section{Euler--Lagrange equation and regularity}

Let
\[
\mathcal V
:=
\{\varphi\in H^{1/2}(\T):\ell_0(\varphi)=0,\ \ell_2(\varphi)=0\}.
\]

\begin{proof}[Proof of Theorem~\ref{thm:EL}]
Let \(\Gamma_\ast\) be the unique minimiser. For every \(\varphi\in\mathcal V\) and
\(\varepsilon\in\R\), the perturbed function \(\Gamma_\ast+\varepsilon\varphi\) belongs to
\(\mathcal A\). Hence the map
\[
\varepsilon\mapsto J(\Gamma_\ast+\varepsilon\varphi)
\]
has a minimum at \(\varepsilon=0\).

Expanding the energy gives
\begin{align*}
J(\Gamma_\ast+\varepsilon\varphi)
&=
\bigl\langle
(-\Delta)^{1/2}(\Gamma_\ast+\varepsilon\varphi),
\Gamma_\ast+\varepsilon\varphi
\bigr\rangle \\
&=
J(\Gamma_\ast)
+
\varepsilon\Bigl(
\langle (-\Delta)^{1/2}\varphi,\Gamma_\ast\rangle
+
\langle (-\Delta)^{1/2}\Gamma_\ast,\varphi\rangle
\Bigr)
+
\varepsilon^2J(\varphi).
\end{align*}
Differentiating at \(\varepsilon=0\), we obtain
\[
\langle (-\Delta)^{1/2}\varphi,\Gamma_\ast\rangle
+
\langle (-\Delta)^{1/2}\Gamma_\ast,\varphi\rangle
=0,
\qquad\forall\varphi\in\mathcal V.
\]
By symmetry, Lemma~\ref{lem:half-lap-symmetric},
\[
\langle (-\Delta)^{1/2}\Gamma_\ast,\varphi\rangle=0,
\qquad\forall\varphi\in\mathcal V.
\]

Thus \((-\Delta)^{1/2}\Gamma_\ast\) belongs to the annihilator of
\[
\mathcal V=\ker\ell_0\cap\ker\ell_2.
\]
Since \(\ell_0\) and \(\ell_2\) are linearly independent continuous linear functionals, the
subspace \(\mathcal V\) has codimension \(2\). Consequently its annihilator in
\(H^{-1/2}(\T)\) is two-dimensional and spanned by \(\ell_0\) and \(\ell_2\). Hence there
exist constants \(\lambda_0,\lambda_2\in\R\) such that
\[
(-\Delta)^{1/2}\Gamma_\ast=\lambda_0\ell_0+\lambda_2\ell_2.
\]
Equivalently,
\[
(-\Delta)^{1/2}\Gamma_\ast=\lambda_0+\lambda_2x^2
\qquad\text{in }H^{-1/2}(\T).
\]
\end{proof}

We next analyse the regularity of the minimiser. Since \(x^2\) is interpreted as a
\(2\pi\)-periodic function, the right-hand side of the Euler--Lagrange equation is not smooth at
\(\pm\pi\). Thus one should not expect the minimiser to be smooth on the torus. The correct
regularity is obtained from the decay of the Fourier coefficients.

\begin{proof}[Proof of Theorem~\ref{thm:regularity}]
Write
\[
\Gamma_\ast(x)=\sum_{k\in\Z}\widehat{\Gamma_\ast}(k)e^{ikx},
\qquad
F(x):=\lambda_0+\lambda_2x^2=\sum_{k\in\Z}\hat F(k)e^{ikx}.
\]
Taking Fourier coefficients in the Euler--Lagrange equation yields
\[
|k|\,\widehat{\Gamma_\ast}(k)=\hat F(k),
\qquad k\in\Z.
\]

The periodic function \(x^2\) has the Fourier expansion
\[
x^2=\frac{\pi^2}{3}+4\sum_{k=1}^{\infty}\frac{(-1)^k}{k^2}\cos(kx),
\]
so that
\[
\hat F(0)=\lambda_0+\lambda_2\frac{\pi^2}{3},
\qquad
\hat F(k)=\frac{2\lambda_2(-1)^k}{k^2}
\qquad (k\neq 0).
\]
Hence
\[
|\hat F(k)|\lesssim |k|^{-2}
\qquad (k\neq 0),
\]
and therefore
\[
|\widehat{\Gamma_\ast}(k)|
=
\frac{|\hat F(k)|}{|k|}
\lesssim |k|^{-3}
\qquad (k\neq 0).
\]

Let \(s<5/2\). Then
\[
\sum_{k\in\Z}(1+|k|^2)^s|\widehat{\Gamma_\ast}(k)|^2
\lesssim
\sum_{k\neq 0}|k|^{2s}|k|^{-6}
=
\sum_{k\neq 0}|k|^{2s-6}.
\]
Since \(2s-6<-1\) when \(s<5/2\), this series converges. Hence
\[
\Gamma_\ast\in H^s(\T)
\qquad\text{for every }s<\frac52.
\]

The Sobolev embedding theorem in one dimension implies that if
\(s>3/2\), then
\[
H^s(\T)\hookrightarrow C^{1,\alpha}(\T)
\qquad\text{for every }\alpha<s-\frac32.
\]
Fix \(\alpha<1\). Choosing \(s\) so that
\[
\frac32+\alpha<s<\frac52,
\]
we conclude that
\[
\Gamma_\ast\in C^{1,\alpha}(\T).
\]
\end{proof}

To pass from the weak equation to a pointwise one on the torus, we use the following regularity
fact.

\begin{lemma}\label{lem:frac-lap-regularity}
Let \(0<\alpha<1\). If \(\Gamma\in C^{1,\alpha}(\T)\), then
\[
(-\Delta)^{1/2}\Gamma\in C^\alpha(\T).
\]
\end{lemma}

\begin{proof}
See Theorem~1.4 in \cite{Roncal}. Taking \(k=1\) and \(\sigma=1\) yields the claim.
\end{proof}

\begin{proposition}\label{prop:EL-pointwise}
Let \(\Gamma_\ast\) be the unique minimiser. Then there exist constants
\(\lambda_0,\lambda_2\in\R\) such that
\[
(-\Delta)^{1/2}\Gamma_\ast(x)=\lambda_0+\lambda_2x^2
\qquad\text{for all }x\in\T.
\]
\end{proposition}

\begin{proof}
By Theorem~\ref{thm:regularity}, \(\Gamma_\ast\in C^{1,\alpha}(\T)\) for every \(\alpha<1\).
Hence, by Lemma~\ref{lem:frac-lap-regularity},
\[
(-\Delta)^{1/2}\Gamma_\ast\in C^\alpha(\T).
\]
On the other hand, Theorem~\ref{thm:EL} asserts that
\[
(-\Delta)^{1/2}\Gamma_\ast=\lambda_0+\lambda_2x^2
\qquad\text{in }\mathcal D'(\T).
\]
Since both sides are continuous and coincide as distributions, they coincide pointwise.
\end{proof}

\section{The classical equation on \texorpdfstring{$(-1,1)$}{(-1,1)} and explicit solution}

We now return to the interval \((-1,1)\). Since \(\Gamma_\ast\in C^{1,\alpha}(\T)\), the
principal value integral is well defined pointwise, and the weak Euler--Lagrange equation
upgrades to a classical one.

We first extend Proposition~\ref{prop:J-classical-to-spectral} from smooth to
\(C^{1,\alpha}\)-regular functions.

\begin{proposition}\label{prop:equiv-C1a}
Let \(0<\alpha<1\) and let \(\Gamma\in C^{1,\alpha}(\T)\). Define
\[
(B\Gamma)(x):=\pv\!\int_{-\pi}^{\pi}\frac{\Gamma'(y)}{x-y}\,dy.
\]
Then
\[
\int_{-\pi}^{\pi}\Gamma(x)(B\Gamma)(x)\,dx=\pi J(\Gamma).
\]
\end{proposition}

\begin{proof}
The proof is the same as in Proposition~\ref{prop:J-classical-to-spectral}, once one checks
that the principal value integral defining \(B\Gamma\) is well defined and that Parseval's
identity applies.

Since \(\Gamma\in C^{1,\alpha}(\T)\), we have \(\Gamma'\in C^\alpha(\T)\). Hence, for each
fixed \(x\in(-\pi,\pi)\),
\[
\pv\!\int_{-\pi}^{\pi}\frac{\Gamma'(y)}{x-y}\,dy
=
\int_{-\pi}^{\pi}\frac{\Gamma'(y)-\Gamma'(x)}{x-y}\,dy,
\]
because the principal value of \(\Gamma'(x)/(x-y)\) vanishes. Moreover,
\[
\left|\frac{\Gamma'(y)-\Gamma'(x)}{x-y}\right|
\le
[\Gamma']_{C^\alpha}|x-y|^{\alpha-1},
\]
and the right-hand side is integrable since \(\alpha\in(0,1)\). Thus \(B\Gamma(x)\) is well
defined for every \(x\in(-\pi,\pi)\).

As before, with
\[
K(t):=\pv\!\Bigl(\frac1t\Bigr),
\]
viewed as a periodic distribution on \(\T\), one has
\[
B\Gamma=\Gamma'*K
\qquad\text{in }\mathcal D'(\T).
\]
Hence
\[
\widehat{B\Gamma}(k)=\widehat{\Gamma'}(k)\,\widehat K(k)=\pi|k|\,\hat\Gamma(k).
\]

It remains to justify Parseval's identity. Since \(\Gamma\in C^{1,\alpha}(\T)\), certainly
\(\Gamma\in L^2(\T)\). Also,
\[
|B\Gamma(x)|
\le
[\Gamma']_{C^\alpha}\int_{-\pi}^{\pi}|x-y|^{\alpha-1}\,dy
\le C,
\]
so \(B\Gamma\in L^\infty(\T)\subset L^2(\T)\). Therefore Parseval applies:
\[
\int_{-\pi}^{\pi}\Gamma(x)(B\Gamma)(x)\,dx
=
2\pi\sum_{k\in\Z}\hat\Gamma(k)\,\overline{\widehat{B\Gamma}(k)}
=
2\pi^2\sum_{k\in\Z}|k|\,|\hat\Gamma(k)|^2
=
\pi J(\Gamma).
\]
\end{proof}

We may now derive the classical Euler--Lagrange equation on the interval.

\begin{proposition}\label{prop:EL-classical-final}
Let \(\Gamma_\ast\) be the minimiser, rescaled to \((-1,1)\). Then
\(\Gamma_\ast\in C^{1,\alpha}(-1,1)\) for every \(\alpha<1\), and there exist constants
\(\lambda_0,\lambda_2\in\R\) such that
\begin{equation}\label{eq:EL-classical-final}
\pv\!\int_{-1}^{1}\frac{\Gamma_\ast'(y)}{x-y}\,dy
=
\lambda_0+\lambda_2x^2,
\qquad x\in(-1,1).
\end{equation}
\end{proposition}

\begin{proof}
Let \(\Gamma_\ast\) denote the minimiser on \(\T\). By Theorem~\ref{thm:regularity},
\[
\Gamma_\ast\in C^{1,\alpha}(\T)
\qquad\text{for every }\alpha<1,
\]
and by Proposition~\ref{prop:EL-pointwise},
\[
(-\Delta)^{1/2}\Gamma_\ast(x)=\lambda_0+\lambda_2x^2,
\qquad x\in\T.
\]

Since \(\Gamma_\ast\in C^{1,\alpha}(\T)\), the principal value integral defining
\(B\Gamma_\ast\) is well defined pointwise. By Proposition~\ref{prop:equiv-C1a},
\[
B\Gamma=\pi(-\Delta)^{1/2}\Gamma
\]
for every \(\Gamma\in C^{1,\alpha}(\T)\). Applying this to \(\Gamma_\ast\), we obtain
\[
\pv\!\int_{-\pi}^{\pi}\frac{\Gamma_\ast'(y)}{x-y}\,dy
=
\pi(\lambda_0+\lambda_2x^2),
\qquad x\in(-\pi,\pi).
\]

Now define the rescaled function
\[
\gamma_\ast(t):=\Gamma_\ast(\pi t),
\qquad t\in(-1,1).
\]
Then \(\gamma_\ast\in C^{1,\alpha}(-1,1)\) and
\[
\gamma_\ast'(t)=\pi\Gamma_\ast'(\pi t).
\]
Performing the change of variables \(x=\pi t\) and \(y=\pi s\), we obtain
\[
\frac1\pi
\pv\!\int_{-1}^{1}\frac{\gamma_\ast'(s)}{t-s}\,ds
=
\pi\lambda_0+\pi^3\lambda_2t^2.
\]
Multiplying by \(\pi\) and absorbing constants into new coefficients, we arrive at
\[
\pv\!\int_{-1}^{1}\frac{\gamma_\ast'(s)}{t-s}\,ds
=
\mu_0+\mu_2t^2,
\qquad t\in(-1,1).
\]
Renaming \(\gamma_\ast\) as \(\Gamma_\ast\) and \(\mu_0,\mu_2\) as
\(\lambda_0,\lambda_2\), we obtain \eqref{eq:EL-classical-final}.
\end{proof}

We may now solve the singular integral equation explicitly.

\begin{theorem}\label{thm:EL-general-solution}
Let \(\Gamma_\ast\) be the unique minimiser. Then there exist constants \(c_0,a,b\in\R\) such that
\[
\Gamma_\ast(x)=c_0+(a+bx^2)\sqrt{1-x^2},
\qquad x\in(-1,1).
\]
\end{theorem}

\begin{proof}
Define
\[
g(x):=\lambda_0+\lambda_2x^2.
\]
Then \eqref{eq:EL-classical-final} can be written as
\[
\mathcal H(\Gamma_\ast')(x)=\frac1\pi g(x),
\qquad x\in(-1,1),
\]
where \(\mathcal H\) denotes the finite Hilbert transform
\[
(\mathcal Hf)(x):=\frac1\pi\pv\!\int_{-1}^{1}\frac{f(y)}{x-y}\,dy.
\]

We now use Tricomi's inversion formula \cite{KingHilbert}. If \(h\in L^{p'}(-1,1)\) with \(p'>1\), then every
\(L^p\)-solution \(f\) of \(\mathcal Hf=h\) has the form
\[
f(x)=\frac{c}{\sqrt{1-x^2}}
-\frac{1}{\pi\sqrt{1-x^2}}
\pv\!\int_{-1}^{1}\frac{\sqrt{1-y^2}\,h(y)}{x-y}\,dy
\]
for some constant \(c\in\R\).

We apply this with \(h=\frac1\pi g\). Since \(g\) is a polynomial,
\(h\in L^{p'}(-1,1)\) for every \(p'>1\). Moreover, by
Proposition~\ref{prop:EL-classical-final},
\[
\Gamma_\ast\in C^{1,\alpha}(-1,1)
\qquad\text{for every }\alpha<1,
\]
and hence \(\Gamma_\ast'\in L^p(-1,1)\) for every \(1\le p<\infty\). Therefore Tricomi's
formula applies to \(f=\Gamma_\ast'\), yielding
\begin{equation}\label{eq:Gamma-prime-form}
\Gamma_\ast'(x)=\frac{c}{\sqrt{1-x^2}}
-\frac{1}{\pi^2\sqrt{1-x^2}}
\pv\!\int_{-1}^{1}\frac{\sqrt{1-y^2}\,g(y)}{x-y}\,dy.
\end{equation}

Using the standard identities for the finite Hilbert transform,
\[
\mathcal H(\sqrt{1-x^2})=x,
\qquad
\mathcal H(x^2\sqrt{1-x^2})=x^3-\frac{x}{2},
\]
we evaluate the integral term in \eqref{eq:Gamma-prime-form} and obtain
\[
\Gamma_\ast'(x)
=
\frac{c}{\sqrt{1-x^2}}
-\frac{1}{\pi\sqrt{1-x^2}}
\Bigl(\lambda_0x+\lambda_2\bigl(x^3-\tfrac{x}{2}\bigr)\Bigr).
\]
Absorbing constants into new coefficients, this becomes
\[
\Gamma_\ast'(x)=\frac{c_1}{\sqrt{1-x^2}}+\frac{ax+bx^3}{\sqrt{1-x^2}}.
\]

Integrating on \((-1,1)\), we obtain
\[
\Gamma_\ast(x)
=
c_0+c_1\arcsin(x)+(a+bx^2)\sqrt{1-x^2}.
\]

It remains to rule out the \(\arcsin\)-term. Since the minimiser arises by rescaling a periodic
function on \(\T\), its endpoint values agree after transport from \((-\pi,\pi)\) to \((-1,1)\).
On the other hand,
\[
\arcsin(1)\neq\arcsin(-1).
\]
Therefore the term \(c_1\arcsin(x)\) cannot occur, and hence \(c_1=0\). This proves that
\[
\Gamma_\ast(x)=c_0+(a+bx^2)\sqrt{1-x^2}.
\]
\end{proof}

\begin{corollary}\label{cor:regularity-minimiser}
Let \(\Gamma_\ast\) be the minimiser. Then
\[
\Gamma_\ast\in C^\infty(-1,1).
\]
\end{corollary}

\begin{proof}
By Theorem~\ref{thm:EL-general-solution},
\[
\Gamma_\ast(x)=c_0+(a+bx^2)\sqrt{1-x^2}.
\]
Each term in this expression is smooth on the open interval \((-1,1)\).
\end{proof}

\begin{remark}
There is no contradiction between Theorem~\ref{thm:regularity} and
Corollary~\ref{cor:regularity-minimiser}. On the torus, the function \(x^2\) is interpreted as a
\(2\pi\)-periodic function and fails to be smooth at the identification points \(\pm\pi\). One
therefore cannot expect the minimiser to be smooth on \(\T\). After rescaling to the open
interval \((-1,1)\), however, the minimiser is given by an explicit formula and is smooth away
from the endpoints.
\end{remark}

\begin{remark}
The three parameters in the general solution correspond naturally to the three conditions in Prandtl's formulation: the tip condition \(\Gamma_\ast(\pm1)=0\), the prescribed total lift, and
the prescribed second moment. The tip condition removes the constant term, while the remaining two constraints determine \(a\) and \(b\). Thus the variational problem in \(H^{1/2}\) rigorously recovers Prandtl's classical bell-shaped circulation law.
\end{remark}

\begin{corollary}
If the physical tip condition \(\Gamma_\ast(\pm1)=0\) is imposed, then
\[
\Gamma_\ast(x)=(a+bx^2)\sqrt{1-x^2}.
\]
In particular, the variational problem in \(H^{1/2}\) recovers Prandtl's classical bell-shaped circulation profile.
\end{corollary}

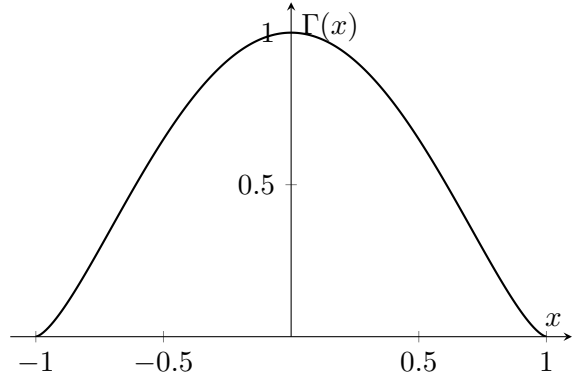
\begin{figure}[h]
\centering
\begin{tikzpicture}
\begin{axis}[
    axis lines=middle,
    xmin=-1.1, xmax=1.1,
    ymin=0, ymax=1.1,
    samples=200,
    xlabel={$x$},
    ylabel={$\Gamma(x)$},
    width=9cm,
    height=6cm
]
\addplot[thick,domain=-1:1] {(1-x^2)*sqrt(1-x^2)};
\end{axis}
\end{tikzpicture}
\caption{A representative bell-shaped circulation distribution.}
\label{fig:bellshape}
\end{figure}

\section*{Statements and Declarations}
The authors declare that they do not have competing interests. 

\end{document}